\documentclass[preprint,12pt]{elsarticle}

\usepackage{amssymb}
\usepackage{amsmath}
\usepackage{amsthm}
\newtheorem{thm}{Theorem}

\newdefinition{rem}{Remark}
\newdefinition{defi}{Definition}
\newproof{pf}{Proof}

\journal{*****}

\begin{document}

\begin{frontmatter}

\title{Non integrability of the $n$ body problem with non zero angular momentum}
\author{Thierry COMBOT\fnref{label2}}
\ead{combot@imcce.fr}
\address{IMCCE, 77 Avenue Denfert Rochereau 75014 PARIS}

\title{Non integrability of the $n$ body problem with non zero angular momentum}

\author{}

\address{}

\begin{abstract}
We prove an integrability criterion and a partial integrability criterion for homogeneous potentials of degree $-1$ which are invariant by rotation. We then apply it to the proof of the meromorphic non-integrability of the $n$ body problem with Newtonian interaction in the plane on a surface of equation $(H,C)=(H_0,C_0)$ with $(H_0,C_0) \neq (0,0)$ where $C$ is the angular momentum and $H$ the energy, in the case where the $n$ masses are equal.
\end{abstract}

\begin{keyword}
Morales-Ramis theory\sep homogeneous potential \sep central configurations \sep differential Galois theory


\end{keyword}

\end{frontmatter}

\section{Introduction}

Non-integrability of homogeneous potentials have already been a lot studied mainly using Morales Ramis theory and its integrability conditions and Ziglin theory. These methods require a particular algebraic orbit of the corresponding potential. With homogeneous potentials, there exist generically straight line orbits, corresponding to the Darboux points of the potentials.

\begin{defi}\label{def0}
Let $V:U\subset \mathbb{C}^n\longrightarrow \mathbb{C}$ be a meromorphic homogeneous function in $q_1,\dots,q_n$. We say that $c\in\mathbb{C}^n$ is a Darboux point if there exists $\alpha\in\mathbb{C}$ such that
$$\frac{\partial}{\partial q_i} V(c)= \alpha c_i \quad\forall i=1\dots n$$
We call $\alpha$ the multiplier, and we say that $c$ is non degenerated if $\alpha\neq 0$. A Darboux point $c$ is also call a central configuration in the case of the $n$ body problem.
\end{defi}

These Darboux points correspond to homothetic orbits, which are explicit algebraic solutions of the differential equations
$$\ddot{q_i}=\frac{\partial}{\partial q_i} V \;\; i=1\dots n$$
With such orbits, it is already possible to prove some facts about non integrability, in particular in the case of homogeneous potentials.

\begin{thm} \label{thm:Morales}(Morales, Ramis, Yoshida \cite{1},\cite{4},\cite{5},\cite{6}) If a meromorphic potential $V$ is meromorphically integrable, then the neutral component of the Galois group of the variational equation near a particular (algebraic) orbit $\Gamma$ is abelian at all order. Moreover, for $V$ homogeneous of degree $-1$ and $\Gamma$ a homothetic orbit associated to a Darboux point $c$ with multiplier $-1$, the Galois group of the first order variational equation has an abelian neutral component if and only if
$$\hbox{Sp}\left(\nabla^2 V(c)\right)\subset \left\lbrace\textstyle\frac{1}{2}(k-1)(k+2),\; k\in\mathbb{N} \right\rbrace.$$
\end{thm}

Here we want to study variational equations and their Galois group near another type of particular orbit that we often encounter when the potential is invariant by rotation. In particular, if there exists a plane of Darboux points, invariant by the rotation symmetry (this case in not rare), then we can build particular orbits with non zero angular momentum. Then we get a one parameter family of orbits for which we can apply Morales Ramis theory. For all of them, the identity component of the Galois group of the variational equation should be abelian, and then we can expect a much stronger integrability criterion than \cite{1}. One difficulty is that the variational equation is too difficult to study in the general case, and then we will only make a complete analysis in the case we will call partially decoupled. We find very strong conditions, only two eigenvalues are possible instead of an infinity.\\

\noindent
The main results of this article are the following.\\

We will first prove that if a homogeneous potential of degree $-1$ invariant by rotation is meromorphically integrable on a surface given by a fixed energy and angular momentum, then the eigenvalue $\lambda$ of the Hessian matrix of a Darboux point with multiplier $-1$ should belong to the following table
\begin{center}
\begin{tabular}{|c|c|}\hline
$C$&$\lambda$\\\hline
$C=0$&$\lambda\in\left\lbrace\textstyle\frac{1}{2}(k-1)(k+2),\;k\in\mathbb{N}\right\rbrace$\\\hline
$C^2H=-1/2$&$\lambda\in\left\lbrace-k^2,\;k\in\mathbb{N}\right\rbrace$\\\hline
$H=0$&$\lambda\in\left\lbrace\textstyle\frac{1}{2}(k-1)(k+2),\;k\in\mathbb{N}\right\rbrace$\\\hline
$(C,H)\in\mathbb{C}^2$&$\lambda\in\{0,-1\}$\\\hline
$(C,H)=(0,0)$&$\lambda\in\mathbb{C}$\\\hline
\end{tabular}\\
\end{center}
where $C$ is the fixed angular momentum and $H$ the fixed energy. The complete statement is Theorem~\ref{thm4} with the table of Theorem~\ref{thm5}, in which there is an additional a priori hypothesis, the ``decoupling condition''. Then we will apply this analysis to a well known case, the $n$ body problem in the plane with equal masses.

\begin{thm}\label{thm15}
The $n$ body problem with equal masses is not meromorphically integrable on any hypersurface of the form $C^2H=\alpha$ with $\alpha\neq 0$ fixed, nor on the hypersuface $H=0$, nor on the hypersurface $C=0$ ($H$ is the energy, and $C$ the angular momentum).
\end{thm}

We see that this type of orbit allows us to study a new type of partial integrability: the case where the potential would be integrable only for a fixed value of the energy and angular momentum. This type of potential exists effectively as given in \eqref{eq13}. This integrability table also gives indications on which particular level we should focus (zero angular momentum, zero energy, and the case $C^2H=-1/2$). This is for example helpful to do brutal search using Hietarinta \cite{7} procedure for finding such potentials. The level $C^2H=-1/2$ is also special because for example in the reduced $3$ body problem we have an additional first integral in this level (the Jacobi integral), although this first integral is not valid everywhere on $C^2H=-1/2$. Theorem~\ref{thm4} cannot solve all problems of this kind because of this ``decoupling condition''. A complete analysis of the $3$ body case gives all the masses which satisfy this condition in Theorem~\ref{thm12}, which are not always symmetric. A non integrability theorem like Theorem~\ref{thm15} is by the way immediate for these masses, except for $(m_1,m_2,m_3)=(1,5,1)$.

\begin{defi}
We will call ``norm'' and scalar product the expressions
$$\lVert v \lVert^2 =\sum\limits_{i=1}^n v_i^2 \qquad <v,w> =\sum\limits_{i=1}^n v_iw_i$$
even for complex $v,w$ (In particular, the ``norm'' can vanish for non zero $v$). We will say moreover that a matrix is orthonormal complex if its columns $X_1,\dots,X_n$ are such that
$$<X_i,X_j>=\sum\limits_{k=1}^n (X_i)_k(X_j)_k=0 \quad\forall i,j\qquad \lVert X_i \lVert^2=\sum\limits_{k=1}^n (X_i)_k^2=1 \quad\forall i$$
We note $\mathbb{O}_n$ the complexified orthogonal group which is the group generated by these matrices, and $\mathbb{SO}_n$ the subgroup of $\mathbb{O}_n$ of matrices with determinant $1$ (corresponding to rotations). In particular, the group $\mathbb{O}_n$ conserve the ``norm''.
\end{defi}

\begin{defi}\label{def1}
Let $V$ be a homogeneous meromorphic potential of degree $-1$ in dimension $n\geq 2$. We note
\begin{equation}
G=\left\lbrace g\in \mathbb{O}_n,\;\;V(g.x)=V(x) \;\forall x\in\mathbb{C}^n\right\rbrace
\end{equation}
We will call $G$ the symmetry group of $V$. We will say that $v\in\mathbb{C}^n$ is in the equator of $G$ if
\begin{equation*}
\{\alpha g.v,\; \alpha\in\mathbb{C},\; g\in G\}
\end{equation*}
contains at least a plane $P$ of dimension $2$ and that $v\in P$. We will say that $V$ is invariant by rotation if $G$ contains at least a subgroup isomorphic to $\mathbb{SO}_2$. We will say that $v$ is an eigenvector of $G$ if for all $g\in G$, $v$ is an eigenvector of $g$.
\end{defi}

\begin{thm}\label{thm1}
Let $V$ be a homogeneous potential of degree $-1$ in dimension $n\geq 2$ and $G$ its symmetry group. Suppose there exists $c$ in the equator of $G$ such that $c$ is a Darboux point of $V$ with multiplier $-1$ and ``norm'' $\lVert c\lVert^2\neq 0$. Then the variational equation near the conic orbit with parameters $(C\lVert c\lVert^2,E\lVert c\lVert^2)$ (angular momentum and energy) is given by
\begin{equation}\label{eq6}
t(-C^2+2t+2Et^2)\ddot{X}+(-t+C^2)\dot{X}=R_{\theta(t)}^{-1} \nabla^2 V(c)R_{\theta(t)} X
\end{equation}
where $R_{\theta(t)}\in\mathbb{SO}_2$ with coefficients in $\mathbb{C}\left(\phi,\sqrt{2E-C^2\phi^{-2}+2\phi^{-1}}\right)$
\end{thm}

\begin{proof}
Let $c$ be a Darboux point of $V$ in the equator of $G$, with multiplier $-1$ and ``norm'' $\gamma=\lVert c\lVert^2\neq 0$. We note $P$ a plane in $G.c$ containing $c$. After rotation of the coordinates, we can suppose that $c=(\gamma,0,\dots,0)$ and that $P$ is generated by $(\gamma,0,\dots,0),(0,\gamma,0,\dots,0)$. A conic orbit for the Darboux point $c$ corresponds to the orbit given by
$$(q_1,q_2)=\varphi_t(1,0)\qquad q_i=0\quad i=3\dots n$$
where $\varphi_t$ is given by
\begin{equation}\label{orbitelliptic}
\varphi_t(x,y)=\phi(t) \left( \begin {array}{cc} cos(\theta(t))&-sin(\theta(t))\\sin(\theta(t))&cos(\theta(t))\\ \end {array} \right)\left( \begin {array}{c} x\\y\\ \end {array} \right)
\end{equation}
Replacing this in the equation of energy conservation and angular momentum (the potential $V$ restricted to the plane $P$ is invariant by rotation), we get
$$\frac{1}{2}\dot{\phi}^2 \gamma^2 +\frac{1}{2}\gamma^2 \phi^2 \dot{\theta}^2+\frac{V(c)}{\phi}=E\gamma^2 \qquad C\gamma^2=\gamma \dot{\theta}$$
And after replacing, we get
$$\frac{1}{2}\dot{\phi}^2 \gamma^2 +\gamma^2\frac{C^2}{2\phi^2}+\frac{V(c)}{\phi}=E\gamma^2$$
Knowing that the multiplier is $-1$, we have using Euler equation for $V$
$$\gamma^2=-V(c) \qquad \frac{1}{2}\dot{\phi}^2 +\frac{C^2}{2\phi^2}-\frac{1}{\phi}=E$$
The variational equation is then of the form
$$\ddot{X}=\frac{1}{\phi(t)^3}\nabla^2V(R_{\theta(t)} c) X$$
with $R_{\theta(t)}$ a rotation matrix. We also know that we are on some conic orbit (due to the fact that the homogeneity degree is $-1$)
$$\phi(t)=\frac{p}{1+e\cos(\theta)}$$
with $p$ and $e$ some parameters depending on $C,E$. We get that $\cos(\theta),\sin(\theta)$ are rational fractions in $\phi,\dot{\phi}$. Then, with variable change $\phi \longrightarrow t$ we get the following expression 
\begin{equation}\label{eq14}
t(-C^2+2t+2Et^2)\ddot{X}+(-t+C^2)\dot{X}=\nabla^2V(R_{\theta(t)} c) X
\end{equation}\newpage
We know that the potential $V$ is invariant by rotation. Then the matrix $\nabla^2V(R_{\theta(t)} c)$ equals to $\nabla^2V(c)$ after a basis change and gives
\begin{equation}\label{eq3}
\nabla^2V(R_{\theta(t)} c)=R_{\theta(t)}^{-1} \nabla^2V(c)R_{\theta(t)}
\end{equation}
Replacing this in \eqref{eq14} gives us the equation \eqref{eq6}.
\end{proof}

\begin{rem}
The main difficulty of this variational equation is that it does not decouple after basis change. Indeed, we can make a basis change with some matrix $P$, but this matrix $P$ should commute with the rotations $R_{\theta(t)}$.
\end{rem}

Let us now give a proper definition of what will call integrable on some level of first integrals, as given in Theorem \ref{thm15}
 
\begin{thm}\label{thmdefine}
Let $V$ be a homogeneous meromorphic potential of degree $-1$ in dimension $n\geq 2$. Let $I_1,\dots,I_k$ be meromorphic first integrals such that 
$$\{I_i,I_j\}=0 \quad \forall i,j$$
where $\{\;,\;\}$ is the Poisson bracket. We pose
$$\mathcal{O}=\left\lbrace(p,q)\longrightarrow g(p,q,I(p,q)) \hbox{ with } g \hbox{ holomorphic on }\mathbb{C}^{2n+k}\right\rbrace$$
the ring of holomorphic functions in $p,q,I$. We suppose that $\;<I_1,\dots,I_k>$ is a prime ideal on $\mathcal{O}$ and we pose $K=\hbox{Frac}(\mathcal{O}/<I_1,\dots,I_k>)$ the corresponding fraction field. Then the following functions are well defined
\begin{itemize}
\item For all $i=1\dots k$, the functions $\varphi_i: K\longrightarrow K,\;\;f \longrightarrow \{f,I_i\}$.
\item The function
$$\Psi:\left(\bigcap\limits_{i=1}^k \varphi_i^{-1}(0)\right)^2\longrightarrow K,\;\;f,g \longrightarrow \{f,g\}$$
\item The functions $K^{n-k}\longrightarrow K$ which associate to $f_1,\dots,f_{n-k}$ a sub determinant of size $n\times n$ of the Jacobian matrix (a matrix of size $2n\times n$) of $I_1,\dots,I_k,f_1,\dots,f_{n-k}$.
\end{itemize}
\end{thm}

\begin{proof}
Let us write a representant of $f\in K$ as $P/Q$, $P,Q\in\mathcal{O}$. We just need to check that the value of the function $\varphi_i$ does not depend on the choice of the representant. We consider $h_1,\dots,h_k,g_1,\dots,g_k\in\mathcal{O}$ and we have
\begin{align*}
\left\lbrace\frac{P+\sum\limits_{s=1}^k h_s I_s}{Q+\sum\limits_{s=1}^k g_s I_s},I_i\right\rbrace=\\
\left(Q+\sum\limits_{s=1}^k g_s I_s\right)^{\!-1}\!\! \left\lbrace P+\sum\limits_{s=1}^k h_s I_s,I_i\right\rbrace-
\frac{P+\sum\limits_{s=1}^k h_s I_s}{(Q+\sum\limits_{s=1}^k g_s I_s)^2}\left\lbrace Q+\sum\limits_{s=1}^k g_s I_s,I_i\right\rbrace=\\
\left(Q+\sum\limits_{s=1}^k g_s I_s\right)^{-1} \{P,I_i\}-\frac{P+\sum\limits_{s=1}^k h_s I_s}{(Q+\sum\limits_{s=1}^k g_s I_s)^2}\{Q,I_i\}=\\
Q^{-1} \{P,I_i\}-PQ^{-2}\{Q,I_i\}=\left\lbrace\frac{P}{Q},I_i\right\rbrace
\end{align*}
so the function is well defined on $K$.

Let us consider $f_1,f_2\in \cap_{i=1}^k \varphi_i^{-1}(0)$ and we write $P/Q$ a representant of $f_1$. Using the fact that $\{I_i,f_2\}=0$, we can do exactly the same calculations as before just replacing $,I_i\}$ by $,f_2\}$. Using the fact that the Poisson bracket is symmetric, we can do the same interverting the indices $1,2$. So the function $\Psi$ is well defined.

Let us consider $x$ one of the variables $p,q$, $f\in K$ and $P/Q$ a representant. We have the classical formula
$$\partial_x \left( \frac{P}{Q}\right)=Q^{-1} \partial_x P -P Q^{-2} \partial_x Q$$
So when we add to $P$ and $Q$ elements of $<I_1,\dots,I_k>$, we are adding in the determinant a linear combination of $\partial_x I_1,\dots,\partial_x I_k$. Using the fact that the determinant is multilinear, this will not change the value of the determinant because it contains the columns of the derivatives of $I_i,\;i=1\dots,k$.
\end{proof}

One need to be extremely cautious when manipulating these derivatives, because $K$ is not a differential field, so we cannot conclude directly that all notions we will need (Poisson brackets, independence) are well defined. For example, the Jacobian matrix itself we consider is not well defined on $K$, only its sub determinant of size $n\times n$ are. Remark that for the following, we will always consider a representant and forget the field $K$, and so this complicated definition will have no impact. This theorem is here to give a proper definition of integrability on some particular level of first integrals, and this complicated presentation has some advantages as it includes also all singular levels thanks to the prime ideal condition (we never ask for example that the first integrals $I_i$ be independent).

\begin{defi}\label{def2}
Let $V$ be a homogeneous meromorphic potential of degree $-1$ in dimension $n\geq 2$. Let $I_1,\dots,I_k$ be meromorphic first integrals such that $\{I_i,I_j\}=0 \quad \forall i,j$
We say that $V$ is meromorphically integrable on the manifold $(I_1,\dots,I_k)=0$ if there exists functions $F_1,\dots,F_{n-k}\in K$ ($K$ is defined as in Theorem \ref{thmdefine}) such that
$$\{H,F_i\}=0\in K \;\;\forall i \qquad \{I_i,F_j\}=0\in K \;\;\forall i,j \qquad \{F_i,F_j\}=0 \;\;\forall i,j$$
and such that at least one of the sub determinants of size $n\times n$ of the Jacobian matrix of $I_1,\dots,I_k,F_1,\dots,F_{n-k}$ is not $0$ in $K$ (this corresponds to the condition of independence almost everywhere).
\end{defi}

\begin{thm}\label{thm3}
Let $V\neq 0$ be a homogeneous meromorphic potential of degree $-1$ in dimension $n\geq 2$. Suppose that an "angular momentum"
$$C=\sum\limits_{k\leq i>j\leq 0 } a_{i,j} (p_i q_j-p_j q_i)\qquad\quad  a_{i,j}\in\mathbb{C}$$
is a non trivial first integral of $V$. Let us fix the value of energy $H=H_0\neq 0$ and angular momentum $C=C_0 \neq 0$. If $V$ is integrable on this manifold of codimension $2$, then $V$ is integrable on the hypersurface $C^2H=C_0^2H_0$.
\end{thm}

\begin{proof}
We consider the following transformation
\begin{equation}\label{eq4}
\varphi\;\; \mathbb{C}^{2n} \longrightarrow \mathbb{C}^{2n} \qquad (p,q)\longrightarrow (\alpha p,\alpha^{-2}q)
\end{equation}
We see that the transformation $\varphi$ just multiply the Hamiltonian $H\longrightarrow \alpha^2 H$, and this does not change the integrability of $H$. Let us suppose that $H$ be integrable on the manifold $(H,C)=(H_0,C_0)$. We have
$$H(\varphi(p,q))=\alpha^2 H \qquad C(\varphi(p,q))=\alpha^{-1} C \qquad (C^2H)(\varphi(p,q))=C^2H$$
Then $H$ is also integrable on the manifold $(H,C)=(\alpha^2 H_0,\alpha^{-1}C_0)$. We also have
$$\bigcup_{\alpha\in\mathbb{C}^*} (\alpha^2 H_0,\alpha^{-1}C_0) =\left\lbrace (p,q)\in\mathbb{C}^{2n} \hbox{ such that } C(p,q)^2H(p,q)=C_0^2H_0 \right\rbrace $$
because $C_0^2H_0\neq 0$. This gives the theorem.
\end{proof}

\begin{rem}
We can see that the study of integrability on a specific manifold make sense only if this manifold is invariant by $\varphi$, because if it is not the case, then our potential will be integrable on a manifold with higher dimension. Remark that the ideals $<C-C_0,H-H_0>,\;<C^2 H-C_0^2H_0>$ are always prime for $C_0^2H_0\neq 0$ and $V\neq 0$, so integrability on these manifolds is well defined, contrary to the case $<C^2H>$ which will need to be splitted in two parts $<C>,<H>$.
\end{rem}

\section{Integrability table}

\begin{thm}\label{thm4}
Let $V$ be a homogeneous meromorphic potential of degree $-1$ in dimension $n\geq 2$ and $G$ its symmetry group. Suppose it exists $c$ in the equator of $G$ such that $c$ is a Darboux point of $V$ with multiplier $-1$ and $\lVert c\lVert^2\neq 0$. We pose $P$ the plane associated to $c$ and $\tilde{G}$ the group of rotations which stabilize the plane $P$. Let $E$ be an eigenspace of $\nabla^2V(c)$ invariant by $\tilde{G}$. If $V$ is meromorphically integrable (respectively for some specific level of energy and angular momentum $(H,C)$), then the following equation possess a Galois group whose identity component is abelian (respectively for some specific level of energy and angular momentum $(H,C)$)
\begin{equation}\label{eq5}
t(-C^2+2t+2Ht^2)\ddot{X}+(-t+C^2)\dot{X}=\lambda X \qquad H,C,\lambda\in\mathbb{C}
\end{equation}
where $\lambda$ is the eigenvalue of $\nabla^2V(c)$ associated to the eigenspace $E$.
\end{thm}

\begin{proof} This is a direct application of Theorem~\ref{thm1}. We have a plane $P$ and all vectors in this plane are Darboux points. The potential restricted to this plane is invariant by rotation (because the Darboux point $c$ is in the equator). On the eigenspace $E$, the matrix $\nabla^2 V(c)$ corresponds to $\lambda Id$. We know moreover that the space $E$ is invariant by the rotations $R_{\theta(t)}$ which corresponds to elements of $\tilde{G}$. We have then
$$\left.{R_{\theta(t)}^{-1} \nabla^2 V(c)R_{\theta(t)}}\right|_E=\lambda Id$$
So the equation \eqref{eq6} on the eigenspace $R$ simplifies and becomes equation \eqref{eq5}.
\end{proof}

\begin{rem}
The Theorem~\ref{thm4} has lots of hypotheses, but in fact only one of them is really restrictive. The existence of an invariant plane $P$ such that the potential is invariant by rotation on this plane is common in practical cases. This often results by symmetry of the system. This is for example always the case in the $n$ body problem. The restrictive condition is the existence of $E$ invariant by rotation. In fact, this is a condition very similar to the codiagonalization constraint from Maciejewski, Przybylska they found studying potentials which are the sum of two homogeneous potentials. In fact, a potential invariant in dimension $n$ by rotation can also be reduced to become a potential in dimension $n-1$ which will be a sum of a homogeneous potential and the potential $C^2/r^2$. This new potential is not homogeneous and our condition correspond to the commutation of the Hessian matrices (at least on some non trivial subspace).
\end{rem}

\begin{thm}\label{thm5}
We consider the differential equation
\begin{equation}\label{eq7}
t(-C^2+2t+(C^2-2)t^2)\ddot{X}-(t-C^2)\dot{X}=\lambda X \qquad C,\lambda \in\mathbb{C}
\end{equation}
This is a Fuchsian equation with $4$ singularities, of Heun type \cite{19} which corresponds to \eqref{eq5} with $H=\frac{1}{2}C^2-1$. The corresponding Galois group is $G =SL_2(\mathbb{C})$ except for the following values of $(C,\lambda)$

\begin{center}

\begin{tabular}{|c|c|}\hline
$C$&$\lambda$\\\hline
$C=0$&$\lambda\in\left\lbrace\textstyle\frac{1}{2}(k-1)(k+2),\;k\in\mathbb{N}\right\rbrace$\\\hline
$C=1$&$\lambda\in\left\lbrace-k^2,\;k\in\mathbb{N}\right\rbrace$\\\hline
$C=\sqrt{2}$&$\lambda\in\left\lbrace\textstyle\frac{1}{2}(k-1)(k+2),\;k\in\mathbb{N}\right\rbrace$\\\hline
$C\notin\{0,1,\sqrt{2}\}$&$\lambda\in\{0,-1\}$\\\hline
\end{tabular}\\
\end{center}

\end{thm}

\begin{proof}
We begin by the case $C\notin\{0,1,\sqrt{2}\}$. The equation \eqref{eq7} has exactly $4$ regular singularities on
$$0,1,\frac{C^2}{2-C^2},\infty$$
We make Frobenius expansion on these $4$ singularities, and we find a logarithmic term for $t=0$ and for $t=\infty$. More precisely, we get
\begin{align*}
X(t)=c_1t^2\left(1-\frac{\lambda-2}{3C^2}t+\frac{\lambda^2-11\lambda+18+6C^4-12C^2}{24C^4}t^2+O(t^3) \right)+\\
c_2\left(\ln\;t\left(\frac{\lambda+\lambda^2}{C^4}t^2+O(t^3)\right)-2-\frac{2}{C^2}\lambda t+O(t^2) \right)\\
X(t)=c_1\left(1+\frac{\lambda}{(C^2-2)t}+O(t^{-2})\right)+\\
c_2\left(t\ln\left(\frac{1}{t}\right)\left(\frac{1+\lambda}{(C^2-2)t}+O(t^{-2})\right)+t\left(1-\frac{1}{(C^2-2)t}+O(t^{-2})\right)\right)
\end{align*}
These expansions are valid for $\lambda\neq -1,0$. In the case $\lambda=-1$, we can compute explicitly the solutions and we find
$$X(t)=c_1(t-C^2)+c_2\sqrt{(t-1)(tC^2-2t+C^2)}$$
The Galois group is then $\mathbb{Z}/2\mathbb{Z}$, abelian. In the case $\lambda=0$, we find the solution
\begin{align*}
X(t)=c_1+c_2\left(\sqrt{(t-1)(tC^2-2t+C^2)}\right)-\\
\frac{c_2}{\sqrt{C^2-2}}\ln\left(\frac{(C^2-2)t+1}{\sqrt{C^2-2}}+\sqrt{(t-1)(tC^2-2t+C^2)} \right)
\end{align*}
The identity component of the Galois group is then $\mathbb{C}$, abelian. Let us consider the case $\lambda\neq -1,0$. Among the  $3$ solvable cases of Kovacic's algorithm, the only possible one with a logarithmic term is when there is a solution of the form
$$X(t)=\exp\left( \int F(s) ds\right)\qquad F\in\mathbb{C}(t)$$
If $F$ has singularities of order more than $2$ then $X$ does not have a Puiseux expansion near this singularity. Impossible because all singularities are regular. If the degree of $F$ is positive, then the expansion at infinity is not a Puiseux series. Then the particular solution $X(t)$ should be of the following form
$$X(t)=\prod\limits_{i=1}^k (t-t_i)^{m_i}$$
If $m_i$ is not a non-negative integer, then $t_i$ is a singularity of $X$ then $t_i$ equals to one of the singularities of the equation. This give even more constraints on the $m_i$ because the Frobenuis exponents on $1,\frac{C^2}{2-C^2}$ are $0,1/2$. On $0$, the possible exponent is $2$, and on infinity it is $0$ (the other ones correspond to the logarithmic behavior). This imply that the sum of the $m_i$ be zero. But the $m_i$ can never be negative, then they are all zero. The only left possibility is then $X(t)=1$. We replace and we find $\lambda=0$, case already done. Then the Galois group is $SL_2(\mathbb{C})$.

\bigskip

The cases $C\in\{0,1,\sqrt{2}\}$ correspond to confluences. These confluences are all regular (this has probably something to do with the fact that the system comes from a variational equation of a Hamiltonian system). The case $C=0$ has already been treated by Morales, Ramis, Yoshida \cite{1},\cite{4}. Let us study the case $C=\sqrt{2}$. This corresponds to the parabolic case (some study of this case has already been done by Tsygvintsev \cite{9}). The equation \eqref{eq7} becomes
$$2t(t-1)\ddot{X}-(t-2)\dot{X}=\lambda X$$
There is a logarithmic term for the singularity $t=0$
\begin{align*}
X(t)=c_1t^2\left(1+\left(\frac{1}{3}-\frac{1}{6}\lambda\right)t+O(t^2) \right)+\\
c_2\left(ln\;t\left(\left(\frac{1}{4}\lambda^2+\frac{1}{4}\lambda\right)t^2+O(t^3)\right)-2-\lambda t-\left(\frac{1}{2}\lambda+\frac{1}{4}\right)t^2+O(t^3)\right)
\end{align*}
for $\lambda\neq 0,-1$. In the cases $\lambda\in\{0,-1\}$, we find the solutions
$$X(t)=c_1+c_2\sqrt{t-1}(2+t)\qquad\quad X(t)=c_1(t-2)+c_2\sqrt{t-1}$$
The Galois group is then $\mathbb{Z}/2\mathbb{Z}$ in both cases, then abelian. We now look at the case $\lambda\neq 0,-1$. The possibles exponents are $\{2\}$ at $0$, $\{0,1/2\}$ at $1$ and
$$\left\lbrace-\frac{3}{4}+\frac{1}{4}\sqrt{8\lambda+9},-\frac{3}{4}-\frac{1}{4}\sqrt{8\lambda+9}\right\rbrace$$
at $\infty$. As before, we prove that we need a solution of the form
$$X(t)=\prod\limits_{i=1}^k (t-t_i)^{m_i}$$
The possibles exponents outside infinity are always integers or half integers, all non-negative, then the sum of the $m_i$ is a non-negative integer or half integer. Then
$$-\frac{3}{4}+\frac{1}{4}\sqrt{8\lambda+9}=\frac{1}{2} (k-1) \qquad k\in\mathbb{N}^*$$
We solve this equation and we find
$$\lambda=\frac{1}{2}(k-1)(k+2)\qquad k\in\mathbb{N}^*$$
This is exactly the condition of Theorem~\ref{thm5}. We now want to compute the Galois group for these remaining cases. We write the solutions of the equation in the following form (it is a hypergeometric equation, and the solutions can be written using hypergeometric series $\vphantom{H}_2F_1$)
\begin{align*}
X(t)=c_1\;\vphantom{H}_2F_1\left(\left[1-\frac{1}{2}k,\frac{1}{2}k+\frac{3}{2}\right],\left[\frac{1}{2}\right],-t+1\right)t^2+\\
c_2\;\vphantom{H}_2F_1\left(\left[2+\frac{1}{2}k,\frac{3}{2}-\frac{1}{2}k\right],\left[\frac{3}{2}\right],-t+1\right)\sqrt{t-1}t^2
\end{align*}
These hypergeometric series are finite if the first bracket in $\vphantom{H}_2F_1$ contains a non-positive integer. For $k\geq 2$, we see that either $1-\frac{1}{2}k$ or $\frac{3}{2}-\frac{1}{2}k$ is a non-positive integer. Then one of the two functions is a polynomial. We always have a solution in $\mathbb{C}[t,\sqrt{t-1}]$, and then the identity component of the Galois group is abelian.

\bigskip

\noindent
Let us now study the case $C=1$. The equation \eqref{eq7} becomes
$$-t(t-1)^2\ddot{X}-(t-1)\dot{X}=\lambda X $$
The expansion on $0$ is the following
\begin{align*}
X(t)=c_1t^2\left(1+\left(\frac{1}{3}-\frac{1}{6}\lambda\right)t+\left(\frac{1}{96}\lambda^2-\frac{11}{96}\lambda+\frac{3}{16}\right)t^2+O(t^3)\right)+\\
c_2\left(\ln\;t\left(\left(\frac{1}{4}\lambda^2+\frac{1}{4}\lambda\right)t^2+O(t^3)\right)-2-\lambda t-\left(\frac{1}{2}\lambda+\frac{1}{4}\right)t^2+O(t^3)\right)
\end{align*}
and possess a logarithmic term for $\lambda\neq 0,-1$. The expansion at infinity is
\begin{align*}
X(t)=c_1\left(1-\frac{\lambda}{2t}+\frac{\lambda (\lambda-5)}{12t^2}+O(t^{-3})\right)+\\
c_2\left(\ln\; t \left(\lambda+1-\frac{\lambda(\lambda+1)}{2t}+O(t^{-2})\right)+t+1-\frac{4+11\lambda+3\lambda^2}{4t}+O(t^{-2})\right)
\end{align*}
Then there is always at least one logarithmic term for $\lambda\neq -1$. Remark that we already now that this equation has an abelian Galois group for $\lambda=0,-1$ (either using the limiting process of the generic solution for all angular momentum, or running Kovacic's algorithm for these specific cases). So now we will suppose that $\lambda\neq 0,-1$. We know that if the Galois group is not $SL_2(\mathbb{C})$, then it exists a solution of the form
$$X(t)=\exp\left( \int F(s) ds\right)\qquad F\in\mathbb{C}(t)$$
The equation is Fuchsian and then $X(t)$ can be written
$$X(t)=\prod\limits_{i=1}^k (t-t_i)^{m_i}$$
The $m_i$ need to be non-negative integers except maybe at singularities. The exponents at $1$ are $+\sqrt{-\lambda},-\sqrt{-\lambda}$. Then one of the following equation is satisfied
$$2+\sqrt{-\lambda}+k=0 \quad\hbox{or} \quad2-\sqrt{-\lambda}+k=0 \qquad k\in\mathbb{N}$$
Then
$$\lambda=-(k+2)^2 \quad k\in\mathbb{N}$$
We add the cases $\lambda=0,-1$ and this gives exactly the condition given by Theorem~\ref{thm5}. We now need to compute the Galois group for these specific cases.  We write the solutions of the equation in the following form (it is a hypergeometric equation, and the solutions can be written using hypergeometric series)
\begin{align*}
X(t)=c_1\;\vphantom{H}_2F_1\left(\left[2-i\sqrt{\lambda},1-i\sqrt{\lambda}\right],\left[1-2i\sqrt{\lambda}\right],1-t\right)t^2(t-1)^{-i\sqrt{\lambda}}+\\
c_2\;\vphantom{H}_2F_1\left(\left[1+i\sqrt{\lambda},2+i\sqrt{\lambda}\right],\left[1+2i\sqrt{\lambda}\right],1-t\right)t^2(t-1)^{i\sqrt{\lambda}}
\end{align*}
These hypergeometric series are finite if the first bracket in the hypergeometric series $\vphantom{H}_2F_1$ contains a non-positive integer. We see that for $\lambda=-k^2\;\; k\in\mathbb{N}^*$, it is the case for the solution in $c_1$. There is always a polynomial solution and then the Galois group is always abelian. Remark that such a work can also be done using Kimura classification of hypergeometric equation which are solvable by quadrature in \cite{14}.

\end{proof}

\section{Algebraic potentials}

In the following sections, we will often need to consider algebraic potentials instead of meromorphic ones. This is a problem because Theorem~\ref{thm:Morales} deals only with meromorphic potentials. This problem is often not addressed, except in Ziglin \cite{15}, but in fact his procedure does not work. This is because making cuts in the complex plane does not allow after to make all possible monodromy paths. Then, the monodromy group will be reduced. It could have no consequences, but here there are important consequences because we absolutely need to be able to turn around the point $0$ in the variational equation (this is because for the two other singularities, the exponents are $0,1/2$, and then if we restrict ourselves to these ones, the monodromy group will always be abelian). Let us now make a precise statement

\begin{defi} Let $w$ be an algebraic ``function''. We define the critical set by
$$\Sigma(w)=\{x,\;\;w\hbox{ is not }C^\infty \hbox{ on } x\}$$
\end{defi}

\begin{thm} \label{thm:algebraic}
Consider an algebraic potential $V$ of degree $-1$ which is integrable with first integrals meromorphic in the positions, impulsions, and $V$, and a non degenerate Darboux point $c$ outside the critical set of $V$. Then the identity component of the Galois group of the variational equation along the homothetic orbit associated to $c$ is abelian at all order. It is also the case for an elliptic orbit if it exists under the conditions of Theorem~\ref{thm1}.
\end{thm}

\begin{proof}
There are two important arguments to apply Morales Ramis theorem
\begin{itemize}
\item The first integral need to have an expansion in series (or the quotient of two series in the meromorphic case) along the curve
\item The coefficients of this expansion will be functions of the time $t$, and the corresponding field will be the base field to be considered in Galois group computation.
\end{itemize}
We have that $V$ is homogeneous, then so is the critical set $\Sigma(V)$. If $c\notin \Sigma(V)$, then the whole orbit $(c.\dot{\phi},c.\phi)$ is not in $\Sigma(V)$, except maybe for $\phi=0$. It is also the case for an elliptic orbit using the notation~\eqref{orbitelliptic} in proof of Theorem~\ref{thm1}. So, a first integral meromorphic in the positions, impulsions, and $V$, is meromorphic everywhere on the particular orbit except possibly for $\phi=0$. Let us prove now that the coefficients of an initial form of such first integral have coefficients meromorphic in $\phi,\sqrt{2E-C^2\phi^{-2}+2\phi^{-1}}$.

Consider first the homothetic orbit $\mathcal{O}$. To compute such an initial form, we need to derivate and evaluate the first integral $I$ on $\mathcal{O}$. We can write by definition
$$I(p,q)=f(p,q,V(q)) \qquad \hbox{ with } f \hbox{ meromorphic}$$
The problem is of course the derivation in $V$ on $\mathcal{O}$. We know it is derivable everywhere on $\mathcal{O}$ except maybe for $\phi=0$. But we also have that the derivatives of $V$ are homogeneous functions. So we have
$$\partial^k V(c.\phi(t))=\beta \phi(t)^{-1-k} \qquad \beta\in\mathbb{C}$$
where $\partial^k$ correspond to a derivation of order $k$ (and $\beta$ depend of course of the chosen derivation). We see that even if there is a singularity on $\phi=0$, the function is still meromorphic on $\phi=0$. The coefficient $\beta$ is well defined because the derivative is well defined for $\phi\neq 0$. In the case of elliptic orbits, the potential is invariant by rotation on the plane generated by the elliptic orbit. So the only problematic point is also $(0,0)$. As given in Theorem~\ref{thm1}, the functions $\cos(\theta),\sin(\theta)$ of the angle are functions meromorphic in $\phi,\sqrt{2E-C^2\phi^{-2}+2\phi^{-1}}$ and then, so are the coefficients of the initial form.

\bigskip

To conclude, we now need to precise that, as Morales and Ramis using Kimura table \cite{14} have done, we are in fact computing Galois group over the base field $\mathbb{C}(\phi)$, not on the field of meromorphic functions. This is not a problem here because in all cases, the variational equation is regular at infinity, and then so are the first integral. We know that a meromorphic function on $\bar{\mathbb{C}}$ is in fact rational, so the coefficients of the initial form are in fact in
$$\mathbb{C}\left(\phi,\sqrt{2E-C^2\phi^{-2}+2\phi^{-1}}\right)$$
This is just an extension of degree $2$ of the field $\mathbb{C}(\phi)$, so the identity component of the Galois group will be the same.
\end{proof}

\section{The case of dimension $3$}

\noindent
In the particular case of dimension $3$, we get

\begin{thm}\label{thm6}
Let $V(x,y,z)$ be a potential meromorphic in $\sqrt{x^2+y^2},$\\ $\sqrt{x^2+y^2+z^2},z^2$ and homogeneous of degree $-1$ in dimension $3$ (this imply that the symmetry group of $V$ contains $\mathbb{Z}/2\mathbb{Z} \times \mathbb{O}_2$). Suppose that $V(1,0,0)\neq 0,\infty$. If $V$ is meromorphically integrable, then it belongs to one of the following families
\begin{equation}\label{eq8}
V=\frac{a}{\sqrt{x^2+y^2+z^2}}\;\;a\in\mathbb{C}^*\qquad V=\frac{b}{\sqrt{x^2+y^2}}\;\;b\in\mathbb{C}^*
\end{equation}
\end{thm}

This theorem is almost the best we can have (with a reasonable statement). To apply our previous theory, we need an invariant plane on which the potential is invariant by rotation. Such invariant plane comes here from the symmetry in $z$. The constraint $V(1,0,0)\neq \infty$ cannot be removed, but the constraint $V(1,0,0)\neq 0$ could maybe be removed with a lot of additional work. There are two keys which allow us to give such a complete statement, which are the fact that the decoupling condition is always satisfied, and then that the potential can be reduced on a plane for which an almost complete classification is already done in \cite{13} (for a finite number of eigenvalues).

\begin{proof}
The potential $V$ possess a symmetry group $G$ such that
$$G\supset \left<\left( \begin {array}{ccc} cos(\theta)&-sin(\theta)&0\\sin(\theta)&cos(\theta)&0\\0&0&1 \end {array} \right),\left( \begin {array}{ccc}1&0&0\\0&1&0\\0&0&-1 \end {array} \right) \right> $$
We consider $P$ the plane $z=0$. This is an invariant plane because $\partial_z V(x,y,0)$ $=0$ thanks to parity in $z$. Using the hypotheses, the restriction of $V$ to the plane $P$ is not zero or infinite. The point $c=(1,0,0)$ is then a non degenerated Darboux point, and $c$ is in the equator of $G$ because $c\in P$ and $G.c\supset P$. The matrix $\nabla^2V(c)$ contains a stable subspace of dimension $2$ associated to $P$. Then the supplementary space generated by the vector $(0,0,1)$ is also an eigenspace. The rotation group generated by the rotations around the $z$-axis let invariant the vector $(0,0,1)$. The conditions of Theorem \ref{thm4} are satisfied and the "vertical" (normal to the plane $P$) variational equation is then
$$t(-C^2+2t+2Et^2)\ddot{X}+(-t+C^2)\dot{X}=\partial_{zz} V(c) X$$
This equation is integrable for all angular momentum $C$ only if
$$\partial_{zz} V(c)\in\{0,-1\}$$
with $c$ with multiplier equal to $-1$. We now restrict our potential to the plane $\tilde{P}:\; y=0$. The potential $V$ is invariant by rotation around the $z$-axis, then $\tilde{P}$ is an invariant plane and we consider the restriction $\tilde{V}:\tilde{P}\mapsto \bar{\mathbb{C}}$.\\
The restriction of the function $\sqrt{x^2+y^2}$ to $y=0$ gives a bivaluated function whose values are $+x$ or $-x$. We can choose one or another for the restriction $\tilde{V}$ (because $\tilde{V}$ should be integrable for both possibilities anyway), and we choose arbitrary
$$\sqrt{x^2+y^2}\mid_{y=0}=x$$
The function $\tilde{V}(x,z)$ is then meromorphic in $x,\sqrt{x^2+z^2},z$. It possess a Darboux point $c=(1,0)$, and it is non degenerated. Using \cite{13}, and then applying the symmetry group, we find that if $V$ is meromorphically integrable, then it should be of the form \eqref{eq8}. These potentials effectively possess an additional first integral, $I=p_z x-p_x z$ and $I=p_z$ respectively, and they are independent with energy and angular momentum almost everywhere.
\end{proof}

\begin{rem}
We can see here the importance of the symmetry group structure in the study of integrability. Here, the vertical variational equation is simple because in dimension $3$, a group of rotations (except $\mathbb{SO}_3$) always possess an eigenvector. This is no more the case in dimension $4$ and higher. In particular, the complexity of the variational equation is very linked to the symmetry group, and if it is too complicated, we will need additional properties on the matrix $\nabla^2 V(c)$. As we will see after, in the $n$ body problem, an explicit decoupling condition appear because the symmetry group contains the rotations
$$R_{\theta}=\left( \begin {array}{cc} \cos\theta I_n&-\sin\theta I_n\\\sin\theta I_n&\cos\theta I_n\\ \end {array} \right)$$
and this group do not possess a common eigenvector of eigenvalue $1$.
\end{rem}

Now we can ask if the Theorems \ref{thm4},\ref{thm5} are really ``useful'' and not only purely theoretical possibilities with no examples. Does exist effectively some potentials that would be integrable only for a specific value of energy and angular momentum? Using Hietarinta \cite{7} direct method and then our non integrability approach, we find the following potentials

\begin{thm}\label{thm7}
We consider the potentials
\begin{equation}\label{eq13}
V_1=\frac{\sqrt{x^2+y^2}}{x^2+y^2-z^2} \qquad V_2=\frac{x^2+y^2+z^2}{(x^2+y^2)^{3/2}}
\end{equation}
The potential $V_1$ is integrable for zero angular momentum $C=0$, but not on any other hypersurface $C^2H=\alpha,\;\;\alpha\in\mathbb{C}^*$ (the question about integrability on $H=0$ is still open).\\
The potential $V_2$ is integrable on the hypersurface $H=0$ of zero energy, but not on any other hypersurface $C^2H=\alpha,\;\;\alpha\in\mathbb{C}^*$, nor on the hypersurface $C=0$ of zero angular momentum.
\end{thm}

\begin{proof}
The first integral of $V_1$ is given by
$$I_1=\frac{(xp_x+yp_y)p_z}{\sqrt{x^2+y^2}}-\frac{z}{x^2+y^2-z^2}$$
The potential $V_1$ possess a Darboux point $(1,0,0)$, and the associated eigenvalue is $\lambda=2$. Using integrability table of Theorem~\ref{thm5}, we have that this value is only possible for the hypersurfaces $C=0$ and $H=0$. Then $V_1$ is not meromorphically integrable on any hypersurface of the form $C^2H=\alpha,\;\;\alpha\in\mathbb{C}^*$.\\
The first integral of $V_2$ is given by
$$I_2=(x^2+y^2-z^2)^2 p_z^2-4z(x^2+y^2-z^2)p_z(xp_x+yp_y)+4z^2(xp_x+yp_y)^2$$
The potential $V_2$ possess a Darboux point $(1,0,0)$, and the associated eigenvalue is $\lambda=2$. Using integrability table of Theorem~\ref{thm5}, we have that this value is only possible for the hypersurfaces $C=0$ and $H=0$. We know it is integrable for $H=0$. Suppose it is integrable for $C=0$. Then we could reduce the potential by rotation and we would obtain the following potential
$$\tilde{V_2}=\frac{q_1^2+q_2^2}{q_1^3}$$
This potential possess a Darboux point $(1,0)$ and the associated eigenvalue is $\lambda=2$. But in this case, it already has been proved in \cite{13} that the potential should belong to one of the following families (after rotation)
\begin{equation}\label{eq1}
V=\frac{a}{q_1}+\frac{b}{q_2}\quad a,b\in\mathbb{C}^*\quad V=\frac{a(q_1^2+q_2^2)}{(q_1+\epsilon iq_2)^3}+\frac{a}{q_1+\epsilon i q_2}\quad a\in\mathbb{C}^*,\;\epsilon=\pm 1
\end{equation}
The second case is impossible because it is always complex. For the first one, we apply a rotation to $\tilde{V_2}$ of angle $\theta$
$$\tilde{V_2}_{\theta}=\frac{q_1^2+q_2^2}{\left(\cos(\theta) q_1+\sin(\theta) q_2\right)^3}$$
and this never coincide with expression \eqref{eq1}. Then $V_2$ is not integrable on the hypersurface $C=0$.
\end{proof}

\section{Application to the $n$ body problem}
We consider $V$ the potential of the $n$ body problem in the plane
\begin{equation}\label{eq9}
V=\sum\limits_{i>j} \frac{m_im_j}{\lVert q_i-q_j \lVert }
\end{equation}
with positive masses $m_i$, $q_i\in\mathbb{C}^2$. The symmetry group is (at least)
$$\left<\left( \begin {array}{cc} \cos\theta I_n&-\sin\theta I_n\\\sin\theta I_n&\cos\theta I_n\\ \end {array} \right),\;\;\theta\in\mathbb{C}\right>$$
Let $c$ be a Darboux point with multiplier $-1$ and such that $\lVert c\lVert^2\neq 0$. Then $c$ is in the equator of $G$ and we can build a conic orbit (by definition, the mutual distances between the bodies are not zero). For the following, we will pose
\begin{equation}\label{eq2}
W_{i,j}=\frac{1}{m_i}\frac{\partial^2}{\partial q_i\partial q_j} V(c)\qquad\quad W\in M_{2n}(\mathbb{C})
\end{equation}
using notation $m_{i+n}=m_i$. Remark for the following that the potential of the $n$ body problem as given by \eqref{eq9} is not reduced at all. This means in particular that the kinetic part is 
$$\sum\limits_{i=1}^n \frac{\lVert p_i \lVert ^2}{2m_i}$$
and so does not correspond exactly to the case we studied before. Still, it is almost the same and we just have to make a variable change like $p_i \longrightarrow p_i \sqrt{m_i}$. The matrix $\nabla^2 V(c)$ becomes in particular the matrix given by \eqref{eq2}.

\subsection{General properties}
\begin{defi}\label{def4}
Let $V$ be the potential of the $n$ body problem with positive masses $m_i$, $c$ a Darboux point with multiplier $-1$. We will say that the variational equation near a conic orbit is partially decoupled if it exists a non trivial vector space $\tilde{V}$ and $\lambda\in\mathbb{C}$ such that
$$Wv=\lambda v \quad \forall v\in\tilde{V}$$
and $\tilde{V}$ is stable by the rotations
$$R_\theta=\left( \begin {array}{cc} \cos\theta I_n&-\sin\theta I_n\\\sin\theta I_n&\cos\theta I_n\\ \end {array} \right)$$
\end{defi}

\begin{rem}
This definition exactly corresponds to the existence of a non trivial eigenspace satisfying Theorem~\ref{thm4}.
\end{rem}

\begin{thm}\label{thm8}
Let $V$ be the potential of the $n$ body problem with positive masses $m_i$, $c$ a Darboux point with multiplier $-1$ and $W\in M_{2n}(\mathbb{C})$ the associated matrix (given by equation \eqref{eq2}). The variational equation near a conic orbit is partially decoupled if and only if it exists a vector $v\in\mathbb{C}^{2n}\setminus \{0 \}$ and $\lambda\in\mathbb{C}$ such that
\begin{equation}\label{eq10}
W v=J^{-1} W J v= \lambda v
\end{equation}
where $J\in M_{2n}(\mathbb{C})$ is matrix of the canonical symplectic form.
\end{thm}

\begin{proof}
Suppose at first that $v$ is not an eigenvector of $R_{\theta}$ (these matrices commute so they have the same eigenvectors). We just have to take $\tilde{V}=(v,Jv)$ because the space generated by $R_{\theta} v,\;\forall \theta$ is a $2$-dimensional space which contains $(v,Jv)$ ($\tilde{V}$ is always $2$-dimensional because $v$ is not an eigenvector of $J=R_{\pi/2}$). Using the hypotheses, $v$ and $Jv$ are eigenvectors of $W$ with the same eigenvalue, so $\tilde{V}$ is an eigenspace of $W$ stable by the rotations $R_{\theta}$. If $v$ is an eigenvector of $R_{\theta}$, then we take $\tilde{V}=\mathbb{C}.v$ and $\tilde{V}$ is an eigenspace of $W$ stable by the rotations $R_{\theta}$.\\
Conversely, if we have an eigenspace $\tilde{V}$ stable by the rotations $R_{\theta}$, we take any vector $v\in\tilde{V}$ and it satisfy \eqref{eq10} because $J=R_{\pi/2}$ and then $Jv\in\tilde{V}$, and so it is also an eigenvector of eigenvalue $\lambda$.
\end{proof}

\begin{thm}\label{thm9}
Let $V$ be the potential of the $n$ body problem with positive masses $m_i$, $c$ a Darboux point with multiplier $-1$ and $W\in M_{2n}(\mathbb{C})$ the associated matrix (given by equation \eqref{eq2}). If the variational equation near a conic orbit is partially decoupled then the matrix $W$ of \eqref{eq2} has a double eigenvalue.
\end{thm}

\begin{proof}
If $dim(\tilde{V}) \geq 2$ then by definition the matrix $W$ has a double eigenvalue. Let us consider the case $dim(\tilde{V})=1$. The corresponding vector have to be a common eigenvector of $J$ and $W$. The eigenvectors $J$ are of the form $(w,iw),\;w\in\mathbb{C}^n$. In particular, they have zero ``norm''. But if $W$ has only simple eigenvalues, then $W$ is diagonalizable and using Theorem 6 of \cite{18}, $W$ is then diagonalizable in an ``orthonormal'' basis. So if $W$ has an eigenvector with zero ``norm'', then this eigenvector is a linear combination of two eigenvectors and this implies an eigenspace of dimension greater than $2$ and then a double eigenvalue.
\end{proof}

\begin{thm}\label{thm10}
Let $V$ be the potential of the $n$ body problem with positive masses $m_i$, $c$ a Darboux point with multiplier $-1$ \textbf{such that the bodies are aligned} and $W\in M_{2n}(\mathbb{C})$ the associated matrix (given by equation \eqref{eq2}). We suppose that $W$ is diagonalizable. Then the variational equation near a conic orbit has a Galois group $G$ such that
$$G\sim\tilde{G}\hbox{ with } \tilde{G}\subset \mathbb{C}\times Sp(2)^{n-2}$$
where $Sp(2)$ is the $4$ dimensional symplectic group.
\end{thm}

\begin{proof}
For an \textbf{aligned} Darboux point, we have the following property (found by direct computation)
\begin{equation}\label{eq12}
W=\left( \begin {array}{cc} A&0\\0&-\frac{1}{2}A\\ \end {array} \right)\qquad J^{-1}W J=\left( \begin {array}{cc} -\frac{1}{2}A&0\\0&A\\ \end {array} \right)
\end{equation}
Then $W$ and $J^{-1}W J$ commute. Then it exists a common eigenvector basis of $W$ and $J^{-1}W J$. Then there exists a decomposition in space $V_i$ of dimension $2$ with the $V_i$ stable by rotations $R_\theta$. We can then write the variational equation under the following form
$$t(-C^2+2t+2Et^2)\ddot{X}+(-t+C^2)\dot{X}=R_{\theta(t)}^{-1} A_iR_{\theta(t)} X \quad i=1..n$$
with $A_i$ a $2\times 2$ matrix (we can choose $A_i$ diagonal after a basis change). Among the matrices $A_i$, there is one corresponding to the motion of the center of mass and this gives $A_1=0$. There is also a matrix corresponding to the first integrals of the energy and angular momentum, and this corresponds to $A_2=diag(2,-1)$. The other matrices do not have a priori special properties. Then the Galois group for the cases $i=1,2$ is $\mathbb{C}$, and for the others, it is at most $Sp(2)$.
\end{proof}

\begin{thm}\label{thm11}
Let $V$ be the potential of the $n$ body problem with positive masses $m_i$, $c$ an \textbf{aligned} Darboux point with multiplier $-1$ and $W\in M_{2n}(\mathbb{C})$ the associated matrix (given by equation \eqref{eq2}). The variational equation near a conic orbit is partially decoupled if and only if $\hbox{det}(W)=0$.
\end{thm}

\begin{proof}
For an aligned Darboux point, we have the equalities \eqref{eq12}. We pose $v=(w_1,w_2)$. If $v$ is an eigenvector of $W$, then $w_1$ is an eigenvector of $A$ and $-\frac{1}{2}A$ with the same eigenvalue. Then $\hbox{det}(A)=0$. Conversely, if $\hbox{det}(W)=0$, then it exists an eigenvector $w$ of eigenvalue $0$ of $A$, and then $v=(w,w)$ is admissible.
\end{proof}

\subsection{The $3$ body problem and some specific cases}

We already know that in all cases, the matrix $W$ should have a double eigenvalue. Our approach will be the following. We search masses and Darboux points such that $W$ has a double eigenvalue. Then for the corresponding eigenvector $v$, there are two possibilities.
\begin{itemize}
\item Either $Jv$ is also an eigenvector of $W$ with the same eigenvalue. This corresponds to the case where the associated eigenspace is of dimension greater than $2$.
\item Either $v$ can be written $v=(w,iw)$, and the matrix $W$ is not diagonalizable.
\end{itemize}
For the aligned case, it is easier because we just have to look at the determinant. But in fact for the real ones, there are never zero eigenvalue if the Darboux point is real (this is due to the result of \cite{10}), so we need to look at complex cases. But even there this constraint is much stronger than expected. We find the following theorem

\begin{thm}\label{thm12}
Let $V$ be the potential of the $3$ body problem with positive masses $m_1,m_2,m_3$ such that $m_1+m_2+m_3=1$. Then $V$ possess a Darboux point such that the variational equation near a conic orbit is partially decoupled if and only if
\begin{equation}\begin{split}\label{eq11}
(m_1,m_2,m_3)=\left( \frac{1}{3},\frac{1}{3},\frac{1}{3} \right),\left( \frac{1}{7},\frac{5}{7},\frac{1}{7} \right),\\
\left(\frac{1}{4}+\frac{\sqrt{21}+\sqrt{126+42\sqrt{21}} } {84},\frac{1}{2}-\frac{\sqrt{21}}{42},\frac{1}{4}+\frac{\sqrt{21}-\sqrt{126+42\sqrt{21}}}{84}\right)
\end{split}\end{equation}
or permutation of these cases.
\end{thm}

\begin{proof}
Let us begin with aligned case. After renormalization, we can take $c=(-1,0,\rho)$ with $\rho\neq 0,-1$ and we have the Euler quintic equation
\begin{align*}
L=\left( -m_1-m_2 \right) \rho^5+ \left( -3m_1-2m_2 \right) \rho^4+ \left( -3\,m_1-m_2\right)\rho^3+\\
\left( 3m_3+m_2 \right) \rho^2+ \left( 3m_3+2m_2 \right) \rho+m_2+m_3=0
\end{align*}
We search the eigenvalues of $W$, and we find that $det(W)=0$ if and only if
$$2\rho^2+3\rho+2=0$$
After taking the resultant, we have
$$Res(2\rho^2+3\rho+2,L,\rho)=7m_2^2-35m_1m_2-35m_2m_3+56m_1^2+63m_1m_3+56m_3^2$$
We want that this resultant vanish, and the only possibility for real positive masses is
$$(m_1,m_2,m_3)=\left( \frac{1}{7},\frac{5}{7},\frac{1}{7} \right) $$
We can permute the masses in the equation and this gives all the possibles permutations of this solution. But there is still a "complex order" and the corresponding potential is the following
$$V={\frac {m_{{1}}m_{{2}}}{q_{{1}}-q_{{2}}}}-{\frac {m_{{1}}m_{{3}}}{q_{{1}}-q_{{3}}}}+{\frac {m_{{2}}m_{{3}}}{q_{{2}}-q_{{3}}}}$$
The Darboux point equation leads to
\begin{align*}
L= \left( -m_{{1}}-m_{{2}} \right) {\rho}^{5}+ \left( -3\,m_{{1}}-2\,m_{{2}} \right) {\rho}^{4}+ \left( -3\,m_{{1}}+2\,m_{{3}}-m_{{2}}\right) {\rho}^{3}+\\
\left( -2\,m_{{1}}+3\,m_{{3}}+m_{{2}} \right) {\rho}^{2}+ \left( 3\,m_{{3}}+2\,m_{{2}} \right) \rho+m_{{2}}+m_{{3}}=0
\end{align*}
The eigenvalues of $W$ never vanish in this case. Let us look now at the Lagrange configuration. For complex coordinates, this corresponds to the case
$$r_1^3=r_2^3=r_3^3$$
where $r_1,r_2,r_3$ are the mutual distances between the bodies. We begin by the case $r_1=r_2=r_3$. We need a double eigenvalue and we find the condition
$$3m_2^2-3m_2m_3-3m_1m_2+3m_3^2-3m_1m_3+3m_1^2=0$$
whose only solution is
$$(m_1,m_2,m_3)=\left( \frac{1}{3},\frac{1}{3},\frac{1}{3} \right)$$
We check that the associated eigenspace of eigenvalue $1/2$ is invariant by $J$, and it is the case.\\
Let us look now at the complex cases. Among the $27-1$ possibilities lots of them are in fact the same after dilatation permutation. After these reductions, we find that there are only $3$ essentially different cases
$$(r_1,r_2,r_3)=(1,1,j),(1,1,j^2),(1,j,j^2)\qquad j=e^{\frac{2i\pi}{3}}$$
The last one is also an aligned Darboux point (it is both Lagrange and Euler configuration), and so it already has been treated.
First we search for masses such that $W$ has a double eigenvalue. We find for $(1,1,j)$ and $(1,1,j^2)$ only one real positive solution, which is the last one of \eqref{eq11}. This is the same for both Darboux points because they are conjugated. We look at the corresponding eigenspace (the double eigenvalue is $1/2$), and we find that the eigenspace is only $1$-dimensional. This is not enough for the case $dim(\tilde{V}) \geq 2$. In the case $dim(\tilde{V})=1$, we know that $W$ should be non-diagonalizable. Moreover, the eigenvector should be written $v$ $v=(w,iw)$. We check these properties and they are satisfied.
\end{proof}

\begin{rem}
The last case of \eqref{eq11} is very interesting for many reasons. We can study the variational equations and the structure of the equations is not so degenerated as in the other cases. Because of this, a more deeper analysis should be possible. For example, in \cite{8},\cite{11}, another notion of partial integrability is considered about the existence of a single additional first integral. For this last masses case, the two notions could probably be fused together to prove the non existence of a single additional first integral restricted to a single level of energy and angular momentum. This is because the variational equation on the characteristic space associated to the eigenvalue $1/2$ is simple enough to allow complete study, but is not trivial. Moreover, the fact that these masses do not possess any symmetry will avoid to consider special invariant sub manifold as the isosceles $3$ body problem in, for example, the complete search of algebraic invariant manifold for the $3$ body problem with these masses.
\end{rem}

\begin{thm}\label{thm13}
We consider $V$ the potential of the $n$ body problem in the plane with positive masses, and $c$ a real Darboux point such that there exist a rotation
$$R_\theta,\;\;\theta\notin \{k\pi ,\; k\in\mathbb{Z} \}$$
in the plane such that $R_{\theta}$ sends the configuration on itself (conserving also the masses). Then there exists a double eigenvalue and the associated eigenspace is of dimension $\geq 2$.
\end{thm}

\begin{proof}
Let $R_{\theta}$ be a rotation such that $\theta\notin \{k\pi ,\; k\in\mathbb{Z} \}$ and that $R_{\theta}$ sends to configuration $c$ on itself and conserve the masses. We pose $W(c)$ as in \eqref{eq2} and we have then the identities
$$W(R_{\theta}c)=R_{-\theta}W(c)R_{\theta}\qquad W(R_{\theta}c)=W(Pc)=P^{-1}W(c)P$$
with $P$ a permutation matrix (the rotation conserve the configuration and the masses of the bodies, but not the numeration of the bodies). Then
$$W(c)=(R_{\theta}P^{-1})^{-1}W(c)R_{\theta}P^{-1}$$
Let $v$ be an eigenvector of $W(c)$. Then $R_{\theta}P^{-1}v$ is also an eigenvector with the same eigenvalue. We just have to prove it is not the same. We can write in a good basis
$$R_{\theta}=\left( \begin {array}{cc} \cos\theta I_n&-\sin\theta I_n\\\sin\theta I_n&\cos\theta I_n\\ \end {array} \right)\qquad 
P=\left( \begin {array}{cc} P_\sigma&0\\0&P_\sigma\\ \end {array} \right)$$
with $P_\sigma$ a permutation matrix. We have then that $P$ and $R_\theta$ commute. We know that the rotation $R_\theta$ is of finite order (because there are only a finite number of bodies and that the configuration is real). Then $\theta=2\pi/k$ with $k\in\mathbb{N}^*$, and $k\geq 3$. The matrix $P$ is then also of order $k$.

Let us consider the body number $i$ with coordinates $q_i$. We look at the orbit $R_\theta^j q_i,\;j=0...k-1$. This orbit contains either $k$ elements or only one (and this case could only happen once, for a body placed on the center of mass). We conclude that the permutation matrix should be of the following form
$$P_\sigma=\left( \begin {array}{cccc} T&0&\dots&0\\0&\dots&0&0\\0&\dots&T&0\\0&\dots&0&1 \end {array} \right)\hbox{ or }P_\sigma=\left( \begin {array}{cccc} T&0&\dots&0\\0&\dots&0&0\\0&\dots&T&0\\0&\dots&0&T \end {array} \right)$$
$$\hbox{with   }T=\left( \begin {array}{cccc} 0&1&\dots&0\\0&0&1&0\\0&\dots&0&1\\1&0&\dots&0 \end {array} \right)$$
We conclude that the matrix $R_\theta P$ can be diagonalized in the form
\begin{align*}
R_\theta P\sim \hbox{diag}\left(e^{i\theta},e^{-i\theta},\left(e^{i(j+1)\theta},\dots,e^{i(j+1)\theta},e^{i(j-1)\theta},\dots,e^{i(j-1)\theta}\right)_{j=0..k-1} \right)\\
\hbox{or   }\quad  R_\theta P\sim \hbox{diag}\left(\left(e^{i(j+1)\theta},\dots,e^{i(j+1)\theta},e^{i(j-1)\theta},\dots,e^{i(j-1)\theta}\right)_{j=0..k-1}\right)
\end{align*}
We suppose that the masses are positive and that the Darboux point is real, then all eigenvectors $v$ of $W(c)$ are real. Suppose that $W(c)$ does not have any eigenspace of dimension $\geq 2$. Then all its eigenvectors are eigenvectors of $R_\theta P$. As $R_\theta P$ is real, if $v$ is a real eigenvector of $R_\theta P$, then the associated eigenvalue is real and so the associated eigenvalue is $\pm 1$. This would mean that
$$Sp(R_\theta P)\subset\{-1,1\}$$
This is impossible because $k\geq 3$.
\end{proof}

\subsection{The equal masses case}

\begin{thm}\label{thm14}
Let $V$ be the potential of the $n$ body problem in the plane with equal masses, $c$ the Darboux point given by the following
$$c_i=\alpha cos\left(\frac{2\pi (i-1)}{n} \right)\qquad c_{i+n}=\alpha sin\left(\frac{2\pi (i-1)}{n} \right)\quad i=1\dots n$$
where $\alpha$ is such that the multiplier equals to $-1$. Let $v$ be the vector given by
$$v_i=\cos\left(\frac{4\pi (i-1)}{n} \right)\qquad v_{i+n}=\sin\left(\frac{4\pi (i-1)}{n} \right)\quad i=1\dots n$$
Then \eqref{eq10} is satisfied with
\begin{equation}\label{eq111}
\lambda=2-\frac{2\sin\left(\frac{\pi}{n} \right)}{1-\cos\left(\frac{\pi}{n} \right)} \left( \sum\limits_{j=1}^{n-1}  \frac{1}{\sin\left(\frac{\pi j}{n}\right)} \right)^{-1}
\end{equation}
\end{thm}

\begin{proof}
The proof is only a direct computation of the matrix $W$ and then of $Wv$ and the use of (lots of) trigonometric formulas.
\end{proof}
We can now eventually prove Theorem \ref{thm15}.

\begin{proof}
Using Theorem~\ref{thm5}, one just need to avoid specific values for $\lambda$. We will then build a majoration and minoration for $\lambda$ given by formula \eqref{eq111}. First of all, we remark that for $n\geq 3$
$$\frac{2\sin\left(\frac{\pi}{n} \right)}{1-\cos\left(\frac{\pi}{n} \right)} \left( \sum\limits_{j=1}^{n-1}  \frac{1}{\sin\left(\frac{\pi j}{n}\right)} \right)^{-1} >0$$
Then $\lambda<2$. Let us prove now that $\lambda>0$. First we prove the following inequality
$$\sin(z)< \frac{1}{\sin(z)} \quad \forall z\in]0,\pi/2[\cup]\pi/2,\pi[$$
and we compute the formula
$$\frac{\sin\left(\frac{\pi}{n} \right)}{1-cos\left(\frac{\pi}{n} \right)}=\sum\limits_{j=1}^{n-1}  \sin\left(\frac{\pi j}{n}\right)$$
Using both of them, this gives for $n\geq 3$
$$\frac{\sin\left(\frac{\pi}{n} \right)}{1-\cos\left(\frac{\pi}{n} \right)} < \sum\limits_{j=1}^{n-1}  \frac{1}{\sin\left(\frac{\pi j}{n}\right)}$$
So we get that $\lambda>0$. Using the integrability table of Theorem~\ref{thm5}, there are no exceptional values in $]0,2[$.
\end{proof}

The case $C^2H=0$ is special. In this case, we have either $C=0$ or $H=0$ (or both). The case $H=0$ corresponds to parabolic orbits. These orbits are used by Tsygvintsev in \cite{9}, and he solves the case for $3$ bodies with equal masses (he studies also existence of a single additional first integral that we do not consider). In the case of $C=0$, the problem is solved by Morales, Simon in \cite{8} for the $n$ equal masses. Our reasoning is also valid for all these cases.

\begin{rem}
The only left case is $H=C=0$. Here, the variational equation is always integrable and in fact it is always the case at all orders. This is linked to the fact that we can reduce the system using homogeneity and rotation, allowing to diminish the dimension of $4$. We obtain then a "direction" field (we lose notion of time after reduction, but not the integrability notions) on a manifold of dimension $4n-8$. This, however, destroy Hamiltonian structure, and moreover, the Darboux points correspond now to fixed points of this field. One would need a new particular orbit (explicit) to apply Morales Ramis method, but no such orbit is known.
\end{rem}

I wish to thank Andrzej Maciejewski and Maria Przybylska who, maybe without even noticing it, help me to enhance this work by their remarks, discussions and invitations in Zielona Gora.

\label{}





\bibliographystyle{model1-num-names}
\bibliography{nonzeroangular}







\end{document}